\documentclass{article}

\usepackage{arxiv}

\usepackage[utf8]{inputenc} 
\usepackage[T1]{fontenc}    
\usepackage{hyperref}       
\usepackage{url}            
\usepackage{booktabs}       
\usepackage{amsfonts}       
\usepackage{nicefrac}       
\usepackage{microtype}      
\usepackage{lipsum}

\usepackage{graphics} 
\usepackage{amsmath} 
\allowdisplaybreaks[4]
\usepackage{amssymb}  
\usepackage{mdwlist}
\usepackage{amsmath}
\usepackage{amsthm} 
\newtheorem{theorem}{Theorem}

\newtheorem{definition}{Definition}
\usepackage{algorithm}
\usepackage{algorithmicx}

\usepackage{graphicx}
\usepackage{epstopdf}
\usepackage{subfig}

\usepackage[title]{appendix} 

\title{\LARGE \bf
Robust Optimization on Unrelated Parallel Machine Scheduling with Setup Times
}

\author{
  David S.~Hippocampus\thanks{Use footnote for providing further
    information about author (webpage, alternative
    address)---\emph{not} for acknowledging funding agencies.} \\
  Department of Computer Science\\
  Cranberry-Lemon University\\
  Pittsburgh, PA 15213 \\
  \texttt{hippo@cs.cranberry-lemon.edu} \\
   \And
 Elias D.~Striatum \\
  Department of Electrical Engineering\\
  Mount-Sheikh University\\
  Santa Narimana, Levand \\
  \texttt{stariate@ee.mount-sheikh.edu} \\
}

\author{

Chutong Gao \\
Department of Mechanics and Engineering Science\\
Peking University, Beijing 100871, China\\ \texttt{oliverpku@pku.edu.cn}\\

\And

Weihao Wang\\
Department of Industrial Engineering and Management\\
Peking University, Beijing 100871, China\\
\texttt{wangweihao@pku.edu.cn}\\

\And
Leyuan Shi \\
Department of Industrial and Systems Engineering\\
University of Wisconsin-Madison, Madison, WI 53706 USA\\
\texttt{leyuan@engr.wisc.edu}\\
}

\begin{document}

\maketitle
\thispagestyle{empty}
\pagestyle{empty}

\begin{abstract}
The parallel machine scheduling problem has been a popular topic for many years due to its theoretical and practical importance. This paper addresses the robust makespan optimization problem on unrelated parallel machine scheduling with sequence-dependent setup times, where the processing times are uncertain, and the only knowledge is the intervals they take values from. We propose a robust optimization model with min-max regret criterion to formulate this problem. To solve this problem, we prove that the worst-case scenario with the maximum regret for a given solution belongs to a finite set of extreme scenarios. Based on this theoretical analysis, the procedure to obtain the maximum regret is proposed and an enhanced regret evaluation method (ERE) is designed to accelerate this process. A multi-start decomposition-based heuristic algorithm (MDH)  is proposed to solve this problem. High-quality initial solutions and an upper bound are examined to help better solve the problem. Computational experiments are conducted to justify the performance of these methods.

\end{abstract}

\section{INTRODUCTION} \label{section1}
With the popularity of mass customization, manufacturing companies are faced with more complex and dynamic customer demands and production environments. The manufacturing industry, like equipment manufacturing, semiconductor fabrication, and medical instrument production, etc., requires effective and efficient production scheduling to cope with the various uncertainties and fulfill the intensive customer needs. 

The problem we discuss in this article is derived from a factory producing high-end equipment in the east of China. The scheduler in this factory needs to schedule a series of R$\&$D jobs to some different labor groups to minimize the maximum completion time of all the groups, i.e., makespan. Each job needs to be finished by exactly one group. Those groups process these jobs in parallel, and each group can only handle no more than one job at a time. All jobs are released from the start, and have no precedence constraints among them. Though any job could be assigned to any machine, the skills and the kind of jobs the groups are adept at vary from each other, so the processing time of one job is non-identical in different groups and cannot be fully correlated by simple rate adjustments. In other words, the processing times at different groups are unrelated. Considering all the factors above, these labor groups could be viewed as unrelated parallel machines. When a machine finishes one job, cleaning, adjustment, and reconfiguration are required to switch to another job. Thus machine-dependent and job sequence-dependent setup times need to be considered in this problem. Our aim is to find the best schedule, i.e., the job assignment and the job sequence on each machine, to optimize the makespan. 

Assuming we know exactly the processing time of each job, this problem could be viewed as the \underline{U}nrelated \underline{P}arallel \underline{M}achine \underline{S}cheduling Problem with Sequence-dependent \underline{S}etup Times (UPMSS) and the objective of makespan minimization, which is a complex combinatorial optimization problem and has been extensively studied. However, this assumption does not quite reflect the actual situations in practice. 
Due to the uncertainty such as fails in product R$\&$D, variations in machine conditions and production environment, the processing time of each job on each machine is uncertain at the time when a scheduler dispatches the jobs. 

Stochastic programming that models the uncertainty as random variables with proper probability distribution is a popular method to deal with the uncertainty in optimization problem (see Leung et al. \cite{leung2006stochastic}). However, due to the lack of historical data and the multiple interdependent factors that lead to the uncertainty, the exact distributions of uncertain parameters are difficult to obtain, and stochastic programming approach with a false distribution assumption may lead to very bad results in reality.

In this factory, the only knowledge is the estimated lower and upper bounds of the processing times based on the historical experience of the scheduler. Under such circumstances, robust approaches are proposed to hedge against the variations of the parameters. 
Two natural ways of defining the possible realisations of uncertain parameters (referred to scenarios) are often used in robust optimization: discrete scenarios and interval scenarios. Discrete scenarios explicitly give the values of the parameters under each specific scenario. In interval scenarios, each parameter can take an arbitrary value from a specified interval. Consider the time uncertainty in the factory we investigate, we adopt the continuous interval scenario in this paper. That is to say, the processing time of each job on each machine is taken from a prespecified interval with a given lower and upper bound.

In robust optimization, two criteria, min-max and min-max regret, are often used to evaluate the robust solutions. For the problem we investigate, the min–max criterion tries to optimize the makespan under the sole worst-case scenario which only consists of the upper bounds of the processing time intervals. This particular worst-case scenario is paid too much more importance. But in reality, it may be unlikely to occur. Therefore, using this criterion in our problem tends to be over-conservative. The min–max regret criterion, however, tries to find a robust solution that has the smallest maximum deviation of the makespan of a given solution from the optimum across all possible scenarios. 
With min-max regret, the decision maker aims to minimize the opportunity loss, i.e., the cost of a solution is compared ex-post to the cost of the best solution which could have been chosen (see Kasperski and Zieli\'nski \cite{kasperski2016robust}), which leads to less conservative decision compared to 
the min-max criteria. However, its difficulty lies in the calculation of the optimal performance under each scenario, which by itself is already very complex, when evaluating each feasible solution. 

To avoid the scheduler's decision from being extremely conservative and make use of the information about both lower bounds and upper bounds of processing times, we choose the min-max regret as the criterion.
Thus this problem be summarized as the \underline{R}obust Optimization on \underline{U}nrelated \underline{P}arallel \underline{M}achine \underline{S}cheduling with Sequence-dependent \underline{S}etup Times and Min-max Makespan Regret Criterion (RUPMSS).

The main contributions of this study are concluded as follows.
\begin{enumerate}
    \item As far as we know, we are the first to consider the uncertainty of the processing times and introduce the robustness with min-max regret criterion to the unrelated machine scheduling with setup times. We propose a robust (min-max regret) makespan minimization problem on unrelated parallel machines with setup times, and the uncertain processing times lie in intervals and no other information on the distribution is known.
    \item The critical properties about worst case scenarios of this problem are proposed to reduce the infinite scenario set to some finite extreme scenarios. Based on it, we give a procedure to obtain the worst-case scenario of a given solution and an \underline{E}nhanced \underline{R}egret 
    \underline{E}valuation method (ERE) to accelerate this procedure.
    \item A \underline{M}ulti-start \underline{D}ecomposition-based \underline{H}euristic Algorithm (MDH) is proposed based on the analysis on the decomposition property and the characteristics of the shift \& interchange search methods. The performance of the algorithm outperforms the existing exact algorithms and heuristics for large cases.
\end{enumerate}

The remainder of this article is organized as follows. 
Section \ref{section2} reviews the literature on the related problems.
In Section \ref{section3}, we formulate RUMPSS as a mixed-integer programming model. Some properties for the worst-case scenarios are proposed and a procedure of finding the worst-case scenario for a given solution is introduced in Section \ref{section4}. In Section \ref{section5}, the ERE method is designed to calculate the maximum regret of a solution. The MDH algorithm with two local search methods is proposed in Section \ref{section6}. Section \ref{section7} is the numerical experiments and Section \ref{section8} is the conclusion and some further directions.



\section{Literature Review} \label{section2}

Robust optimization has been a popular method to hedge against uncertainty in many combination optimization problems for decades, such as the assignment problem \cite{pereira2011exact, wu2018exact}, product packing problem \cite{liu2019product}, vehicle routing problem \cite{lee2012robust,lu2019robust}, etc. More studies on robust optimization problem could refer to the survey papers by Kasperski and Zieli\'nski \cite{kasperski2016robust}, and by Aissi et al. \cite{ aissi2009min}.

The robust optimization has also been used to solve many scheduling problems to achieve certain objectives under various uncertainties, among which many studies have paid attention to the time uncertainty with interval parameters.
For instance, some studies pay attention to the robust project scheduling problem. Artigues et al. \cite{artigues2013robust} study a resource-constrained robust project scheduling problem with uncertain activity duration to minimize the maximum regret of makespan. They propose both an iterative exact algorithm called scenario relaxation method and a heuristic algorithm to solve this problem. Bruni et al. \cite{bruni2017adjustable} consider an adjustable resource-constrained robust project scheduling problem with uncertain activity duration. The task sequencing needs to be decided without violating the task precedence relationship and resource constraints to minimize the worst-case makespan. An exact solution method similar to the Bender's decomposition method is proposed to solve this problem. 

As for the robust production scheduling problem, Kasperski \cite{kasperski2005minimizing} discusses a robust single machine scheduling problem with the objective of lateness. The processing times are specified as intervals and a polynomial algorithm is proposed to solve this problem. Pereira \cite{pereira2016robust} further studies another robust single machine scheduling problem with interval processing time and the objective of total weighted completion time. A branch-and-bound algorithm to solve this problem. Drwal \cite{drwal2018robust} studies a robust single machine scheduling problem with the objective to minimize the weighted number of late jobs. The author assumes the due dates of jobs are uncertain and belong to intervals. They develop a polynomial-time algorithm to solve the special case where all weights are equal to 1. And a MILP formulation is presented for the general case. The solution quality of two simple heuristic algorithms are assessed and a decomposition method is proposed to handle the large-scale problems. Drwal and J\'ozefczyk \cite{drwal2020robust} further study the above robust single machine scheduling problem with interval job processing time and the objective of the weighted number of late jobs, and propose a specialized branch-and-bound algorithm to solve it. 
As for the flow shop, Kouvelis et al. \cite{kouvelis2000robust} show that a two-machine flow shop robust scheduling problem with uncertain processing time and min-max regret criterion is NP-hard. Gholami-Zanjani et al. \cite{gholami2017robust} study a robust flow shop scheduling with interval processing times and sequence-dependent setup times with the objective of min-max weighted flow time. They analyze the robust counterpart of the model and apply it to a real-case in PCB assembly company. Liao and Fu \cite{liao2020min} consider a general permutation flow shop scheduling problem with interval production time. A robust model is formulated with the bi-objectives of min-max regret of total completion time and min-max production tardiness. A directed graph tool is applied to identify the worst-case scenario and a genetic algorithm is implemented to solve this problem.   

Because parallel machines are widely present in various manufacturing systems, some work has also paid attention to the robust parallel machine scheduling problem with min-max regret criterion and interval parameters. 
Several studies focus on the objective of total flow time. Drwal and Rischke \cite{drwal2016complexity}, Xu et al. \cite{xu2014hedge} minimize the maximum regret of total flow time on identical and uniform parallel machine with interval job processing time, respectively.  
Siepak and J\'ozefczyk \cite{siepak2014solution}, Conde \cite{conde2014mip} consider the robust scheduling problem above under unrelated parallel machine settings.

Xu et al. \cite{xu2013robust} consider a robust identical parallel machine scheduling problem with interval job processing times to minimize the maximum regret of the makespan. 
They propose an iterative exact algorithm to solve the small-scale problem. A hill-climbing and a simulated annealing algorithm are further proposed to solve larger scale problems. Feng et al. \cite{feng2016robust} consider the robust scheduling problem of a two-stage hybrid flow shop setting where the first stage is a single machine and the second stage consists of identical parallel machines. They also select the min-max regret of makespan as the objective function and use the same method as Xu et al. \cite{xu2013robust} to solve this problem. However, the aforementioned literature on robust parallel machine scheduling rarely takes the setup time, especially the sequence-dependent setup time, into account.

As for the problem we investigate, a handful of studies have paid attention to the deterministic version of the problem, namely the UPMSS. The parallel machine scheduling problem is NP-hard even for a case of two identical parallel machines \cite{garey1979computers}. 
A series of models such as those in Guinet \cite{guinet1993scheduling}, Avalos-Rosales et al. \cite{avalos2015efficient}, and Fanjul-Peyro et al. \cite{fanjul2019reformulations} are proposed to improve the scale of the instances that could be solved to optimality. 
And some exact algorithms to solve UPMSS, for instance the Decomposition-based Method by Tran et al. \cite{tran2016decomposition} and Mathematical-programming Base Algorithm by Fanjul-Peyro et al. \cite{fanjul2019reformulations}, could be found in literature.  
Nevertheless, due to the NP-hard nature of this problem, various heuristic and meta-heuristic algorithms could be found in the literature. 
For instance, Lin and Ying \cite{lin2014abc} propose a Hybrid Artificial Bee Colony algorithm (HABC) that outperforms the best algorithms at that time such as ACO, TS, RSA, and Meta-RaPS. Wang et al. \cite{wang2016hybrid} propose a Hybrid Estimation of Distribution Algorithm (EDA) with Iterated Greedy (IG) search method and compare it with two GA algorithms. Arnaout \cite{arnaout2020worm} further proposes a Worm Optimization algorithm (WO) for this problem.  
More related studies could refer to the survey paper by Allahverdi \cite{allahverdi2015third}.

To the best of our knowledge, the robust unrelated parallel machine scheduling with sequence-dependent setup times has not been studied by other researchers. So this study intends to fill this gap and solve the problem that could be encountered in many production areas.


\section{Problem Formulation} \label{section3}

We formulate our RUPMSS problem based on Avalos-Rosales et al. \cite{avalos2015efficient} 's model for the deterministic UPMSS. There is a set of jobs $N=\{1,\cdots,n\}$ to be processed on machines $M=\{1,\cdots,m\}$. $p_{ij}^s$, the processing time of job $j$ on machine $i$ under scenario $s\in S$ (the set of all the scenarios) ranges from $[\underline{p}_{ij}, \overline{p}_{ij}]$. After processing job $j$ on machine $i$, a fixed sequence-dependent setup time $s_{ijk}$ is needed if the next job on machine $i$ is job $k$. Our aim is to minimize the maximum regret of a solution $\pi\in \Phi$, i.e., $R_{\max}(\pi) =  \underset{s\in S}{\max} (F(\pi,s) - F^s_*)$, where $F(\pi,s)$ represents the makespan of solution $\pi$ under scenario $s$, $F^s_*$ is the optimal makespan under scenario $s$, and $\Phi$ is the set of all the feasible schedules. 
The notations we use are summarized as follows:

{\bf Notations}
\begin{description}
  \item[$N$] set of jobs, $N=\{1,\dots,n\}$;
  \item[$N_0$] set of jobs with a dummy job 0 added, $N_0 = \{0\}\cup N$
  \item[$M$] set of machines, $M=\{1,\dots,m\}$;
  \item[$j,k$] index of jobs, $j, k\in N_0$;
  \item[$i$] index of machines, $i \in M$;
  \item[$S$] set of scenarios, $S = \left\{[\underline{p}_{ij},\overline{p}_{ij}] | i\in M, j\in N_0 \right\}$;
  \item[$p_{ij}^s$] the processing time of job $j$ on machine $i$ under scenario $s$, $p_{ij}^s\in [\underline{p}_{ij}, \overline{p}_{ij}]$. Note that $\underline{p}_{i0} = \overline{p}_{i0} = 0$ for any machine $i$. In other words, $p_{i0}^s = 0$ for any scenario $s\in S$;
  \item[$s_{ijk}$] the setup time of job $k$ after job $j$ on machine $i$. $s_{i0k}$ is the setup time of job $k$ if it is the first to be processed on machine $i$, and $s_{ij0} = 0$, which means job $j$ is the last to be processed on machine $i$. Let $s_{ijj}=0$ for $\forall j\in N_0$; 
  \item[$V$] a large number.
\end{description}


{\bf Variables}
\vspace{-0.1cm}
\begin{description} 
  \item[$x_{ijk}$]  $x_{ijk}=1$ if job $k$ is the successor of job $j$ on machine $i$; otherwise $x_{ijk}=0$. Note that $x_{i0k}(x_{ik0}) = 1$ means that job $k$ is the first(last) to be processed on machine $i$. $X =\{x_{ijk}|i\in M, j, k\in N_0\}$;
  \item[$y_{ij}$] $y_{ij} = 1$ if job $j$ is assigned to machine $i$. Otherwise, $y_{ij} = 0$. $Y = \{y_{ij}|i\in M, j\in N_0\}$;
  \item[$C_{j}^s$] the completion time of job $j$ under scenario $s$;
  \item[$C_{\max}^s$] the makespan under scenario $s$.
\end{description}

We note that each feasible schedule $\pi\in \Phi$ could be described by $\{X,Y\}$, where $Y$ specifies the assignment of the jobs and $X$ further describes the sequence of the jobs on each machine. It's easy to see that $F(\pi,s) = C_{\max}^s=\underset{j\in N}{\max} ~C_j^s$, and $F^s_* = \underset{\pi\in \Phi}{\min} F(\pi,s)$.

Then, the model could be formulated as follows:
\begin{alignat}{1}
\label{C1} &\underset{\pi\in \Phi}{\min}~\underset{s\in S}{\max}\left( F(\pi,s) - F^s_* \right) \\
\label{C2} \textrm{s.t.}&\sum_{i \in M} y_{i j}=1, \quad j \in N\\
\label{C3} &y_{i j}=\sum_{k \in N_{0}, j \neq k} x_{i j k},\quad i \in M, j \in N\\
\label{C4} &y_{i k}=\sum_{j \in N_{0}, j \neq k} x_{i j k},\quad i \in M, k \in N\\
\label{C5} & \sum_{k\in N} x_{i 0 k} \leq 1,\quad i \in M\\
\label{C6} &C_{k}^s-C_{j}^s+V\left(1-x_{i j k}\right) \geq s_{i j k}+p_{i k}^s,\quad j\in N_{0}, j \neq k, k \in N, i \in M, s\in S  \\
\label{C7} &C_{\max}^s \geq C_{j}^s, \quad j\in N, s\in S\\
\label{C8} &\sum_{j,k \in N_{0} } s_{ijk}x_{i j k} + \sum_{j \in N}p_{i j}^s y_{ij} \leq C_{\max}^s,\quad i \in M, s\in S\\
\label{C9} & C_0^s = 0,\quad s\in S\\
\label{C10} & x_{ijk}\in\{0,1\}, y_{ij}\geq 0, C_j^s\geq 0, C^s_{\max}\geq 0,\quad j, k \in N, i \in M, s\in S
\end{alignat}

Constraints (\ref{C2}) ensure that each job should be assigned to one machine exactly. 
Constraints (\ref{C3}) guarantee that each job has one and only one follow-up job on the machine it is assigned. The last job at the end of each machine would also followed by a dummy job. 
Constraints (\ref{C4}) ensure that each job on each machine is followed by one and only one preceding job. For the first job on the machine, its preceding job is the dummy job. 
Constraints (\ref{C5}) state that at most one jobs is scheduled as the first job after the dummy job on each machine. 
The following constraints are all related to the scenarios the uncertain parameters may occur. 
Constraints (\ref{C6}) are called subtour elimination constraints (SEC) in TSP literature. They stipulate the relationship of the completion time of the jobs assigned to each machine to avoid loops and subtours. 
Specifically, under each scenario $s$, if job $k$ is scheduled after job $j$ on machine $i$, then the job $k$ should be completed in at least $s_{i j k}+p_{i k}^s$ time units after the completion of job $j$.
Constraints (\ref{C7}) stipulate that the makespan should be no less than the completion time of each job in each scenario, which are feasible cuts proved efficient by \cite{avalos2015efficient}.
Constraints (\ref{C8}) are the definition of the makespan. In constraints (\ref{C8}), the left side of the inequality calculate the completion time of each machine as the sum of the sequence-dependent setup time and the processing time of all jobs assigned on it. Then these constraints states that in each scenario the overall makespan should be no less than the completion time of each machine.
The completion times of dummy jobs under all scenarios are set to 0 as shown in constraints (\ref{C9}).
Constraints (\ref{C10}) shows the type and range of each variable. According to the problem structure and constraints (\ref{C3}) and (\ref{C4}), variable $y_{ij}$ is the sum of binary variables $x_{ijk}$ which is bounded by 1 in constraints (\ref{C2}). Therefore it could be relaxed to positive rather than binary variable in this model.

\section{Theoretical Analysis on Worst-case Scenarios} \label{section4}

We could reformulate the nonlinear programming model above as below:
\begin{alignat}{1}
&\label{C11} \underset{\pi}{\min}~ r\\
&\label{C12} F(\pi,s) - F_*^s \leq r,\quad s\in S\\
& \text{Constraints (\ref{C2})-(\ref{C10})} \notag
\end{alignat}

The model is extremely difficult to solve. The underlying difficulties lie in three aspects: (1) The number of the scenarios is infinite, i.e., $|S|=+\infty$, and thus the number of the constraints is infinite; (2) the number of the decision variables is also infinite because of the existence of $C_j^s$, which are necessary for SEC constraints; (3) the calculation of $F_*^s$ requires solving NP-hard deterministic UPMSS under every scenario. Thus, we must examine the structure of the model to capture the characteristics of the worst-case scenarios. In the following analysis, we could limit the scenarios to some finite extreme cases. First, we give the definitions of the critical machine and extreme scenario for a schedule.

\begin{definition}
We denote a machine $i$'s completion time in schedule $\pi$ under scenario $s$ as $F_i(\pi,s)$. A machine $f$ is critical in schedule $\pi$ under scenario $s$ if its completion time is maximum among all machines, i.e.,$ F_f(\pi,s) = \underset{i\in M}{\max}~F_i(\pi,s)$, or equivalently, $F_f(\pi,s)= F(\pi,s)$.
\end{definition}

\begin{definition}
Extreme scenario $s^f, f\in M$ for schedule $\pi$ is defined as:
\begin{equation}
    p_{i j}^{s^{f}}=\left\{\begin{array}{ll}
\overline{p}_{i j}, \text { if } y_{f j}=1 \\
\underline{p}_{i j}, \text { if } y_{f j}=0, i\in M, j\in N_0
\end{array}\right.
\end{equation}\label{equ:extreme}
\end{definition}

With these definitions, we propose Theorem \ref{theorem:extreme} to obtain the worst-case scenarios for a schedule.

\begin{theorem}\label{theorem:extreme}
For a schedule $\pi$, let $s^0$ be a worst-case scenario for $\pi$, in which machine $f$ is a critical machine. Then the extreme scenario $s^f$ for schedule $\pi$ satisfies:

(a) Machines $f$ is also a critical machine in schedule $\pi$ under scenario $s^f$;

(b) Scenario $s^f$ is also a worst-case scenario for $\pi$.
\end{theorem}

\begin{proof}
For a schedule $\pi$, we denote the set of jobs on machine $i$ as $M_i^{\pi}$, i.e., $j\in M_i^{\pi}$ if and only if $y_{i j} = 1$.

Given schedule $\pi$ and scenario $s^0$, we can obtain the critical machine $f$, $f = \underset{i\in M}{ \arg \max} F_i(\pi,s^0)$. Therefore the extreme scenario $s^f$ for schedule $\pi$ is obtained by letting $p_{i j}^{s^f}\leftarrow \overline{p}_{i j}$ for any job $j\in M_f^\pi$ and $i=f$; $p_{i j}^{s^f}\leftarrow \underline{p}_{i j}$, otherwise.

For any schedule $\tilde{\pi} = \{\tilde{x}_{i j k}, \tilde{y}_{i j}|i\in M, j,k\in N_0\}$, 

\begin{equation}
    \begin{array}{l}
         F(\tilde{\pi},s^f) - F(\tilde{\pi}, s^0)
         = \underset{i\in M}{\max}~ F_i(\tilde{\pi}, s^f) - \underset{i^{\prime}\in M}{\max}~ F_{i^{\prime}}(\tilde{\pi},s^0)
        \leq \underset{i\in M}{\max}~ \left\{F_i(\tilde{\pi},s^f) - F_i(\tilde{\pi},s^0) \right\}\\
        = \underset{i\in M}{\max} ~\{ ( \underset{j,k\in N_0}{\sum} s_{i j k}\tilde{x}_{i j k}  + \underset{j\in N}{\sum} p_{i j}^{s^f}\tilde{y}_{i j} )
        - ( \underset{j,k\in N_0 }{\sum} s_{i j k}\tilde{x}_{i j k} + \underset{j\in N}{\sum} p_{i j}^{s^0}\tilde{y}_{i j} ) \}\\
         = \underset{i\in M}{\max}~ \underset{j\in N}{\sum}(p_{i j}^{s^f} - p_{i j}^{s^0}) \tilde{y}_{i j}\\
         = \underset{i\in M}{\max} ~\{ \underset{j\in M_f^{\pi} }{\sum}   (p_{i j}^{s^f} - p_{i j}^{s^0}) \tilde{y}_{i j}
        + \underset{j\notin M_f^{\pi} }{\sum}   (p_{i j}^{s^f} - p_{i j}^{s^0}) \tilde{y}_{i j} \} \notag
    \end{array}{}
\end{equation}{}

Denote:
\begin{equation}
\begin{array}{l}
     \text{SUM1}_i = \underset{j\in M_f^{\pi}}{\sum} (p_{i j}^{s^f} - p_{i j}^{s^0}) \tilde{y}_{i j}\\
     \text{SUM2}_i = \underset{j\notin M_f^{\pi}}{\sum}   (p_{i j}^{s^f} - p_{i j}^{s^0}) \tilde{y}_{i j}\notag
\end{array}
\end{equation}

For the original $\pi = \{x_{i j k}, y_{i j}\}$,
\begin{equation}
\begin{array}{l}
     F(\pi,s^f) - F(\pi,s^0) 
     = \underset{i\in M}{\max} ~\{ \underset{j\in M_f^\pi}{\sum}   (p_{i j}^{s^f} - p_{i j}^{s^0}) y_{i j}
    + \underset{j\notin M_f^\pi}{\sum}   (p_{i j}^{s^f} - p_{i j}^{s^0}) y_{i j} \}\\
    =  \underset{j\in M_f^\pi}{\sum}   (p_{f j}^{s^f} - p_{f j}^{s^0}) y_{f j}
    + \underset{j\notin M_f^\pi}{\sum}   (p_{f j}^{s^f} - p_{f j}^{s^0}) y_{f j}\\
    =  \underset{j\in M_f^\pi}{\sum}   (\overline{p}_{f j} - p_{f j}^{s^0}) \cdot 1
    + \underset{j\notin M_f^\pi}{\sum}   (\underline{p}_{f j} - p_{f j}^{s^0})\cdot 0\notag
\end{array}{}
\end{equation}{}

Denote:
\begin{equation}
\begin{array}{l}
     \text{SUM1} =  \underset{j\in M_f^\pi}{\sum}   (\overline{p}_{f j} - p_{f j}^{s^0}) \cdot 1\\
     \text{SUM2} = \underset{j\notin M_f^\pi}{\sum}   (\underline{p}_{f j} - p_{f j}^{s^0})\cdot 0\notag
\end{array}
\end{equation}

If $i\neq f$, then 
$$\text{SUM1}_i\leq 0 \leq \text{SUM1}$$

If $i = f$, then
\begin{equation}\notag
\begin{array}{l}
    \text{SUM1}_i = \underset{j\in M_f^\pi}{\sum} (\overline{p}_{f j} - p_{f j}^{s^0}) \cdot \tilde{y}_{f j}
    \leq \underset{j\in M_f^\pi}{\sum}   (\overline{p}_{f j} - p_{f j }^{s^0}) \cdot 1 = \text{SUM1}
\end{array}
\end{equation}

So for any $i\in M$, $SUM1_i\leq SUM1$.

Next, for $j\notin M_f^\pi$, $p_{i j}^{s^f} - p_{i j}^{s^0} \leq 0\implies \text{SUM2}_i\leq 0 =\text{SUM2}$.

Now we obtain that 
\begin{equation}
\begin{array}{l}
     \underset{i\in M}{\max}~ (\text{SUM1}_i + \text{SUM2}_i) \leq \text{SUM1} + \text{SUM2}  
     \implies F(\tilde{\pi}, s^f) - F(\tilde{\pi}, s^0) \leq F(\pi, s^f) - F(\pi, s^0)\notag
\end{array}
\end{equation}

Let $\pi^{s^0}_*$ be an optimal solution under scenario $s^0$. So we have 
\begin{equation}
\begin{array}{l}
     F^{s^f}_* - F^{s^0}_* \leq F(\pi^{s^0}_*, s^f) - F(\pi^{s^0}_*, s^0) 
     \leq F(\pi, s^f) - F(\pi, s^0)  
     \implies F(\pi,s^0)-F^{s^0}_* \leq F(\pi,s^f)-F^{s^f}_* \notag
\end{array}
\end{equation}{}

Since $s^0$ is a worst-case scenario for schedule $\pi$, $s^f$ is also a worst-case scenario for $\pi$.
\end{proof}{}

With Theorem \ref{theorem:extreme}, we can limit the set of scenarios $S$ to the finite set of extreme scenarios. According to the way we define extreme scenarios, the size of all possible extreme scenarios is $m\cdot 2^n$ (There are $m$ possible critical machines; for each machine, each job could be on this machine or not, thus $2^n$). However, the critical machine is not known in advance. To evaluate a solution $\pi$, we need to traverse $m$ extreme scenarios $s^1,\cdots,s^m$ for $\pi$ to calculate the maximum regret, i.e., $R_{max}(\pi) = \underset{i\in M}{\max}~\left(F(\pi,s^i) - F^{s^i}_*\right)$. Due to the existence of the nested deterministic UPMSS for the value of $F^{s^i}_*$, it cannot be solved directly by the mixed-integer linear programming method. Therefore, effective methods need to be proposed to solve this problem.

\section{Enhanced Regret Evaluation Method}\label{section5}
With Theorem \ref{theorem:extreme}, we could limit the problem to the extreme scenarios. However, given a solution $\pi'$, the critical machine is not known in advance. To evaluate a solution $\pi'$, we need to traverse $m$ extreme scenarios $s^1,\cdots,s^m$ for $\pi'$ to calculate the maximum regret, i.e., $R_{\max}(\pi') = \underset{i\in M}{\max} \left\{ F(\pi',s^i) - F^{s^i}_*\right\}$. Due to the NP-hard nature of the deterministic UPMSS problem, the evaluation process is also NP-hard, which is computationally expensive because $m$ deterministic problems need to be solved. 

In this section, therefore, we seek to design a framework combining various strategies to accelerate the evaluation process. First, we propose a lower bound of the maximum regret for a neighborhood solution to help judge whether this solution is potential to be better and needs an accurate evaluation.
If a solution needs to be evaluated accurately, we further propose a theorem to reduce the number of machines to be traversed and eliminate some extreme scenarios that cannot be the worst-case scenarios.
And for each extreme scenario to be evaluated, a lower bound of the deterministic UPMSS(which yields an upper bound for the regret) to help judge whether the extreme scenario is possible to be the worst-case scenario and reduce the number of deterministic problem to be solved. After all the procedures above, we adopt the \underline{H}ybrid \underline{A}rtificial \underline{B}ee \underline{C}olony algorithm(HABC) proposed by \cite{lin2014abc} to further accelerate the efficiency of solving deterministic UPMSS. We conclude the framework with these strategies as the \underline{E}nhance \underline{R}egret \underline{E}valuation Method (ERE).

\subsection{Use a lower bound to abandon inferior neighborhood solutions}

Because the neighborhood could be very large and most of the neighborhood solutions would not lead to the improvement, we first develop a lower bound of the maximum regret for a specific neighborhood solution to exclude those inferior solutions which would definitely not improve the solution. 
Assume $\pi^*$ is the current best solution, and $\pi'$ is one of its neighborhood solution. We have
\begin{equation}
    R_{\max}(\pi')=\underset{i\in M}{\max}\{F(\pi',s^i)- F_*^{s^i} \}
    \geq \underset{i\in M}{\max}\{F(\pi',s^i)-F(\pi^*, s^i) \}\notag 
\end{equation}
Because $F(\pi^*, s^i) \geq F_*^{s^i}$, the inequality holds. Thus $R_{LB}(\pi',\pi^*)= \underset{i\in M}{\max}~\left\{ F(\pi',s^i)-F(\pi^*, s^i) \right\}$ is a lower bound for $R_{\max}(\pi')$. Because $\pi^*$ is already known in the neighborhood search process, $R_{LB}(\pi',\pi^*)$ is easy to calculate without calling the MIP solver. If $R_{LB}(\pi',\pi^*)>R_{\max}(\pi^*)$, $\pi'$ cannot be better than $\pi^*$ and there is no need to calculate the exact optimal value of the deterministic problem to evaluate this solution.

\subsection{Eliminate inferior extreme scenarios with Theorem \ref{theorem:5}} 
In the evaluation process of a certain solution, a total of $m$ extreme scenarios need to be traversed.
Each machine is assumed as the critical machine and the corresponding extreme scenario is used to calculate the regret so as to find the maximum regret of this solution.
For each machine, a deterministic UPMSS problem needs to be solved, which makes the evaluation process burdensome and time-consuming. Here we propose a theorem to reduce the number of machines to be traversed in the evaluation so as to accelerate the evaluation process. 

To introduce this theorem, first we need to calculate the machine completion time interval for each machine. Specifically for machine $i\in M$ in solution $\pi = \{X,Y\}$, the machine completion time lies in an interval $[\underline{F_i}(\pi), \overline{F_i}(\pi)]$ under the interval processing time. 
The minimum completion time of machine $i$ over all scenarios is $\underline{F_i}(\pi)=\sum_{j\in N_{0}, k \in N, k \neq j} s_{ijk} x_{ijk}+\sum_{j \in N} \underline{p}_{ij}y_{ij}$ and its maximum completion time is $\overline{F_i}(\pi)=\sum_{j\in N_{0}, k \in N, k \neq j} s_{ijk} x_{ijk}+\sum_{j \in N} \overline{p}_{ij}y_{ij}$. 
With the completion time interval of each machine, we have the following theorem:

\begin{theorem}\label{theorem:5}
For a given solution $\pi$ and machine $i$, if its completion time's upper bound $\overline{F_i}(\pi)$ is lower than $\underline{F_{i'}}(\pi)$ of any other machine $i'\in M, i'\neq i$, $s^i$ cannot be the worst-case scenario of solution $\pi$. 
\end{theorem}

\begin{proof}
 For solution $\pi$, if $s^i$ is the worst-case scenario, machine $i$ should be the critical machine. Thus the completion time on different machines should satisfy $F_i(\pi,s^i) \geq F_{i'}(\pi,s^i)$, for $\forall i'\neq i$. According to (\ref{equ:extreme}), the way we define $s^i$ for solution $\pi$ suggests that $F_{i'}(\pi, s^i) = \underline{F_{i'}}(\pi)$, and $F_i(\pi, s^i) = \overline{F_i}(\pi)$. Since machine $i$ is the critical machine, which means $F_i(\pi,s^i) \geq F_{i'}(\pi,s^i)$, then we have $\overline{F_i}(\pi) \geq \underline{F_{i'}}(\pi)$. Thus we prove the theorem.
\end{proof}

According to Theorem \ref{theorem:5}, in the evaluation process, there is no need to calculate the regret with machine $i$ as the critical machine once we find a machine $i'$ that makes $\overline{F_i}(\pi) < \underline{F_{i'}}(\pi)$. Based on this dominance rule, we could eliminate some machines and corresponding extreme scenarios and reduce the evaluation time for a given solution. We use $S_{\psi}$ to denote the remained extreme scenarios for the subsequent evaluation of this solution.

\subsection{Use an upper bound to avoid inferior extreme scenarios} 

For those extreme scenarios in $S_{\psi}$ for solution $\pi$ that have not be eliminated by the dominance rule, We further use a lower bound of the deterministic UPMSS, which yields an upper bound for the regret of the corresponding scenario, to reduce the number of the deterministic problems to be solved in the evaluation of each extreme scenario.

Specifically, we calculate the regret $R_{\{1\}}(\pi)$ of one extreme scenario $s^{\{1\}} \in S_{\psi}$ first. For any other extreme scenario $s^{\{i\}} \in S_{\psi}$, we calculate a lower bound $LB^{s^{\{i\}}}$ of the deterministic problem under scenario $s^{\{i\}}$ and replace $F^{s^{\{i\}}}_*$ in $R_i(\pi)=F(\pi,s^{\{i\}}) - F^{s^{\{i\}}}_*$ with $LB^{s^{\{i\}}}$. Then we could obtain an upper bound of the maximum regret as $F(\pi,s^{\{i\}}) - LB^{s^{\{i\}}}$. If it is less than the current maximum regret, this extreme scenario cannot be the worst-case scenario. So we don't need to calculate $F^{s^{\{i\}}}_*$ of deterministic RUPMSS under extreme scenario $s^{\{i\}}$, which would further reduce the time to evaluate a certain solution.

We adopt a classical lower bound for the deterministic UPMSS has been used in many literature such as \cite{arnaout2010two}. The lower bound for UPMSS under extreme scenario $s$ is as follows:
\begin{enumerate}
    \item $LB_1^{s} = \frac{1}{M} \sum_{k=1}^{N} \min _{i\in M, j\in N} \{p_{ik}^{s}+s_{ij k}\}$;
    \item $LB_2^{s} = \max_{k\in N} \min_{i\in M, j\in N} \{p_{ik}^{s} + s_{ijk}\}$;
    \item $LB_3^{s}$: the optimal objective value of the LP relaxation of UPMSS.
\end{enumerate}
And $LB^{s} = \max\{LB_1^{s},LB_2^{s},LB_3^{s}\}$.

\subsection{Solve the deterministic problem with heuristics}

Despite all the mechanisms we have designed to accelerate the evaluation process, there remain a number of deterministic UPMSS problems to solve during every evaluation. Due to the NP-hard nature of the deterministic problem, it is very time consuming to obtain the exact solution of it. Thus, we substitute heuristics for the use of solvers like CPLEX to improve the time efficiency and to obtain a high-quality approximate solution. In this study, we choose the HABC algorithm proposed by Lin and Ying\cite{lin2014abc}, which yields a superior performance for UPMSS compared to many other algorithms like ACO, TS, RSA, and Meta-RaPS, etc. The HABC algorithm is a population-based metaheuristic algorithm that mimics the behaviors of a bee colony. It has been successfully applied to a variety of real world-problems including UPMSS. For more details of HABC, the readers could refer to Lin and Ying's work \cite{lin2014abc}.

Based on the methods above, given the current best solution $\pi^*$, the procedure of the enhanced regret evaluation method for a given solution $\pi$ is introduced as follows: 

\begin{eqnarray*}
\hline                               \\
\leftline{\emph{Enhanced Regret Evaluation Method (ERE)} }         \\
\hline
\end{eqnarray*}
\vspace{-0.7cm}
\begin{enumerate}
\setlength{\itemindent}{6pt}
\setlength{\labelsep}{0pt}
  \item[\emph{Step 1.}]\  Set possible extreme scenarios $S_{\psi}=\emptyset$,  $R_{\max}(\pi)=0$;
  \item[\emph{Step 2.}]\ Calculate the lower bound regret of solution $\pi$ as $R_{LB}(\pi,\pi^*) = \max_{i\in M}\{ F(\pi,s^i)-F(\pi^*, s^i)\}$. If $R_{LB}(\pi)>R_{\max}(\pi^*)$, abandon this solution $\pi$, stop.
  \item[\emph{Step 3.}]\ Calculate the machine completion time interval $[\underline{F_i}(\pi), \overline{F_i}(\pi)]$ for each machine $i\in M$. For each machine $i\in M$, if $\overline{F_i}(\pi)<\underline{F_{i'}}(\pi)$ of any other machine $i'\in M, i'\neq i$, move to the next machine; otherwise, add extreme scenario $s^i$ into $S_{\psi}$. If all machines have been fathomed, go to Step 4.
  \item[\emph{Step 4.}]\ Calculate the regret of first extreme scenario $s^{\{1\}}$ in $S_{\psi}$ as $R_{\{1\}}(\pi)=F(\pi,s^{\{1\}}) - F^{s^{\{1\}}}_*$. Here $F^{s^{\{1\}}}_*$ is calculated by the HABC algorithm. For each extreme scenario $s^{\{i\}} \in S_{\psi}, s^{\{i\}}\neq s^{\{1\}}$, calculate the lower bound $LB^{s^{\{i\}}}$ of deterministic UPMSS under scenario $s^{\{i\}}$. If $F(\pi,s^{\{i\}}) - LB^{s^{\{i\}}} \leq R_{\{1\}}(\pi)$, skip this extreme scenario, go to next extreme scenario in $ S_{\psi}$; otherwise calculate $F^{s^{\{i\}}}_*$ with HABC. If $R_{\{i\}}(\pi)=F(\pi,s^{\{i\}}) - F^{s^{\{i\}}}_*>R_{\max}(\pi)$, update $R_{\max}(\pi)\leftarrow R_{\{i\}}(\pi)$. If all scenarios in $S_{\psi}$ have been fathomed, go to Step 5.
  \item[\emph{Step 5.}]\  Terminate the algorithm and output the maximum regret $R_{\max}(\pi)$.
\end{enumerate}
\vspace{-0.2cm}
\begin{eqnarray*}
\hline
\end{eqnarray*}

\section{Multi-Start Decomposition-based Heuristic Algorithm}\label{section6}

In this section, we propose a \underline{M}ulti-start \underline{D}ecomposition-based \underline{H}euristic (MDH) algorithm to solve this problem. 
There are two things to be decided in the RUPMSS: (1) the assignment of jobs to machines (corresponding to decision variables $y_{ij}, \forall i\in M, j\in N$) and (2) the sequencing of the jobs on each machine given the job assignment (corresponding to decision variables $x_{ijk}, \forall i\in M, j, k\in N$). Given multiple initial solutions, the MDH algorithm searches for a better solution with the local search methods in each initial solution's neighborhood. Note that for any solution $\pi$, we use the ERE introduced in the previous section to calculate the maximum regret $R_{\max}(\pi)$. In the following, we investigate the job sequencing subproblem and the job assignment subproblem under RUPMSS, respectively. In addition, we introduce an upper bound of RUPMSS derived from one of the initial solutions of the MDH algorithm. Then we propose the MDH algorithm based on the findings.

\subsection{Job sequencing problem}

We first consider the sequence of jobs on each machine for a given job assignment. Specifically, a fixed job assignment means that for each machine, the jobs to be produced on it are known and unchanged. In this case, the production sequence of jobs on each machine needs to be decided to obtain a complete solution due to the existence of sequence-dependent setup times in RUPMSS, i.e., we need to decide $X$ given a fixed $\hat{Y}$. In this subsection, we'd like to prove that for each fixed assignment, only one optimal job sequence needs to be considered for the optimal solution of RUPMSS. Thus, we could decompose the problem into two stages. 

To give a detailed interpretation of the property we introduce above, first we need to give some new definitions. 

\begin{definition}
For a given job assignment $\hat{Y}$, define schedule $\hat{\pi}^{ \dagger } = \{\hat{X}^{\dagger} , \hat{Y}^{\dagger}\}$ as the optimal solution of the following MILP:
\begin{alignat}{1}
OptSeq(\hat{Y}):~~ & \label{C14} \min ~ \sum_{i\in M} \sum_{j,k\in N_0} s_{ijk} x_{ijk}\\
& \label{C15} y_{ij} = \hat{y}_{ij},\quad i\in M, j\in N\\
& C_k - C_j + V(1-x_{ijk}) \geq s_{ijk},\quad j\in N_0,k\in N, j\neq k, i\in M \\
& C_0 = 0\\
& x_{ijk}\in \{0,1\}, C_j\geq 0,\quad j,k\in N_0, i\in M\\
& \text{Constraints ($\ref{C3}$)($\ref{C4}$)($\ref{C5}$)} \notag
\end{alignat}
The objective ($\ref{C14}$) is the total setup time on all machines.
Constraints ($\ref{C15}$) ensure that the optimal schedule shares the same job assignment as $\hat{Y}$. The other constraints describe a feasible schedule as before. 
\end{definition}

$OptSeq(\hat{Y})$ is meant to find the schedule $\hat{\pi}^{\dagger}$ that minimizes the total setup time under the fixed job assignment $\hat{Y}$. Since the setup times are deterministic and irrelevant to the realisation of scenarios, $OptSeq(\hat{Y})$ is also irrelevant to the scenarios. Then we give the following theorem:

\begin{theorem}\label{theorem:2}
Denote the set of feasible schedules with the job assignment $\hat{Y}$ as $\Phi(\hat{Y})$. $ \arg \min_{\hat{\pi}\in \Phi(\hat{Y})} R_{\max}( \hat{\pi} )$ is obtained by solving $OptSeq(\hat{Y})$. In other words, if $\hat{\pi}^{\dagger}$ is the optimal solution of $OptSeq(\hat{Y})$, then $R_{\max}(\hat{\pi}^{\dagger})\leq R_{\max}(\hat{\pi}), \forall \hat{\pi} \in \Phi(\hat{Y})$.  
\end{theorem}

\begin{proof}
For a fixed job assignment $\hat{Y}$, it is evident that minimizing the total setup time on all the machines is equivalent to minimizing the setup time on each machine separately. Let $s^{0\dagger}$ be the worst-case scenario of the optimal schedule $\hat{\pi}^{\dagger}$ of $OptSeq(\hat{Y})$. For any $\hat{\pi}\in \Phi(\hat{Y})$,
\begin{alignat}{1}
F_i(\hat{\pi}^{\dagger} ,s^{0\dagger} ) & = \sum_{j,k\in N_0} s_{ijk}\hat{x}_{ijk}^{\dagger} + \sum_{j\in N_0} p^{s^{0\dagger} } \hat{y}_{ij} \notag \\
F_i(\hat{\pi} ,s^{0\dagger} ) &  = \sum_{j,k\in N_0} s_{ijk}\hat{x}_{ijk} + \sum_{j\in N_0} p^{s^{0\dagger} } \hat{y}_{ij} \notag\\
\implies & F_i(\hat{\pi}^{\dagger} ,s^{0\dagger} ) \leq F_i(\hat{\pi} ,s^{0\dagger} ),\quad \forall i\in M \notag\\
\implies & \label{C19} F(\hat{\pi}^{\dagger} ,s^{0\dagger} ) \leq F(\hat{\pi},s^{0\dagger} )
\end{alignat}
From ($\ref{C19}$), we could derive the following inequalities:
\begin{equation} \label{C20}
R_{\max}(\hat{\pi}^{\dagger} ) = F(\hat{\pi}^{\dagger}, s^{0\dagger} ) - F_*^{s^{0\dagger}} \leq F(\hat{\pi}, s^{0\dagger} ) - F_*^{s^{0\dagger}} 
= R(\hat{\pi} , s^{0\dagger}) \leq R_{\max}(\hat{\pi})
\end{equation}

\end{proof}

Theorem $\ref{theorem:2}$ gives us significant insight into the structure of RUPMSS: once we obtain a job assignment $\hat{Y}$, $OptSeq(\hat{Y})$ could yield a complete schedule $\hat{\pi}^{\dagger}$ with the smallest maximum regret in any schedules with job assignment $\hat{Y}$. So from now on, we only need to focus on finding the best possible $\hat{Y}$ that has the minimum $R_{\max}(\hat{\pi}^{\dagger})$, where $\hat{\pi}^{\dagger}$ is decided by $\hat{Y}$.

\subsection{Job assignment problem}

Now we come to the decision of job assignment, that is, to assign jobs to the unrelated parallel machines. 

We propose a simple local search method to search for a better job assignment, which includes two operators to conduct the local search: 
\begin{itemize}
    \item \textbf{Shift}: Moving a job from one machine to another one;
    \item \textbf{Interchange}: Swapping two jobs on different machines.
\end{itemize}
These two operators are simple and flexible, and have been used in many problems with similar settings such as the workload balancing problem on parallel machine \cite{cossari2013minimizing,cossari2012new}, the task allocation problem \cite{chen2000hybrid}, and the parallel machine scheduling problem \cite{yilmaz2014genetic,avalos2015efficient}. We use these two operators to search the neighborhood of the incumbent solution sequentially.
If a new assignment with the corresponding optimal sequence introduced earlier has a lower maximum regret, we obtain a better solution and accept it as a new incumbent solution. 

We further discuss under what condition the shift or interchange may be effective. Similar to \cite{xu2013robust}, the following theorems are derived to avoid ineffective search:

\begin{theorem}\label{Theorem:LS}
If a given solution could be improved by adjusting the job assignment, one of the adjusted jobs must be on the critical machine of the incumbent solution under its worst-case scenario. 
\end{theorem}

\begin{proof}
The proof is evident when there are only two machines. We need to consider the case where there are more than two machines.

Consider improving the solution $\pi$ by altering the assignment of jobs on the solution with worst-case scenario $s^0$.
We construct a new solution $\pi'$ by altering the job assignment on its non-critical machines, e.g., by shift or interchange, and we use ${s^{0}}'$ to denote the worst-case scenario for solution $\pi'$. Assume by contradiction that the maximum regret of $\pi'$ is less than that of $\pi$. Then we have
\begin{equation}\label{equation:jobassignment}
R(\pi,s^0) > R(\pi',{s^0}') \geq R(\pi',s^0).
\end{equation}

Given that the adjusted jobs are assigned on non-critical machines, the makespan on the critical machine remains unchanged. Thus the makespan in schedule $\pi'$ cannot be lower than that of $\pi$ under scenario $s^0$, i.e.,
$F(\pi',s^0) \geq F(\pi,s^0)$. Based on this result, we have:
$$
R(\pi',s^0)=F(\pi',s^0)-F^{s^0}_* \geq F(\pi,s^0)-F^{s^0}_*=R(\pi,s^0), \notag
$$
which contradicts (\ref{equation:jobassignment}). Thus an adjustment for the job assignment without adjusting the jobs on the critical machine under the worst-case scenario would not improve the solution. The proof is completed. 
\end{proof}

We note that Theorem \ref{Theorem:LS} could be true for any kind of adjustments besides shift and interchange. But in this problem we only adopt these two. Theorem \ref{Theorem:LS} suggests that for the shift operator, the job we choose to shift must be on the critical machine of the incumbent solution under its worst-case scenario, and for the interchange operator, one of the swapped jobs must come from the critical machine.

\subsubsection{Shift procedure}
For the shift local search, each of the $n$ jobs can be shifted to the other $m-1$ machines, which can generate a total of $(m-1)n$ different assignments in the neighborhood of the initial solution $\pi_0$. But according to Theorem \ref{Theorem:LS}, only shifting the job from the critical machine under the worst-case scenario to other machines may improve the incumbent solution. Therefore only $(m-1)n_c$ assignments adjustment are attempted and evaluated, where $n_c$ is the number of jobs on the critical machine. We denote the $p$th job on the critical machine as $j_{p}$. 
The procedure of shift local search is presented as follows:

\begin{eqnarray*}
\hline                               \\
\leftline{\emph{Shift Procedure} }         \\
\hline
\end{eqnarray*}
\vspace{-0.7cm}
\begin{enumerate}
\setlength{\itemindent}{6pt}
\setlength{\labelsep}{0pt}
\item[\emph{INPUT:}]\ initial solution $\pi_0$, its critical machine $i_c$, maximum regret $R_{max}(\pi_0)$, and the number of jobs on the critical machine $n_c$. Let the current solution $\pi\leftarrow \pi_0$.
  \item[\emph{Step 1.}]\ Set $p\leftarrow 0$.
  \item[\emph{Step 2.}]\ Set $p\leftarrow p+1$, $i\leftarrow 0$. If $p>n_c$ go to Step 6. 
  \item[\emph{Step 3.}]\ Set $i\leftarrow i+1$. If $i>m$, go to Step 2. If $i=i_c$, repeat Step 3. 
  \item[\emph{Step 4.}]\ Construct a new solution $\pi'$ by shifting the job $j_{p}$ to machine $i$ and solving the single machine scheduling problem under any scenario for each machine. Evaluate solution $\pi'$ and obtain its critical machine $i_c'$, maximum regret $R_{\max}(\pi')$, and the number of jobs on critical machine $n_c'$.
  \item[\emph{Step 5.}]\ If $R_{\max}(\pi')<R_{\max}(\pi)$, set $\pi \leftarrow \pi'$, $R_{\max}(\pi)\leftarrow R_{\max}(\pi')$, $i_c \leftarrow i_c'$, $n_c\leftarrow n_c'$, go to Step 1. Otherwise, go to Step 3.
  \item[\emph{Step 6.}]\ Terminate the algorithm and output $\pi$, $R_{\max}(\pi)$, and $i_c$.
\end{enumerate}
\vspace{-0.2cm}
\begin{eqnarray*}
\hline
\end{eqnarray*}

\subsubsection{Interchange procedure}
For the interchange local search, we traverse each pair of jobs ($C_n^2=n(n-1)/2$ in total) to see if it satisfy Theorem \ref{Theorem:LS}. Only when two jobs on different machines and one of them is on the critical machine under the worst case scenario, the interchange is conducted to see whether it could improve the incumbent solution. According to Theorem \ref{Theorem:LS}, only $(n-n_c)n_c$ assignments in interchange neighborhood may bring improvement and are evaluated. 

The procedure of interchange local search is presented as follows:

\begin{eqnarray*}
\hline                               \\
\leftline{\emph{Interchange Procedure} }         \\
\hline
\end{eqnarray*}
\vspace{-0.7cm}
\begin{enumerate}
\setlength{\itemindent}{6pt}
\setlength{\labelsep}{0pt}
  \item[\emph{INPUT:}]\ initial solution $\pi_0$, its critical machine $i_c$, maximum regret $R_{max}(\pi_0)$, and the number of jobs on the critical machine $n_c$. Let the current solution $\pi\leftarrow \pi_0$.
  \item[\emph{Step 1.}]\ Set $l\leftarrow 0$, $k\leftarrow 1$, $j\leftarrow 2$.
  \item[\emph{Step 2.}]\ Set $l\leftarrow l+1$. If $l>\frac{(n-1)n}{2}$ go to Step 7. 
 \item[\emph{Step 3.}]\ If $j>n$, set $k\leftarrow k+1$. If $k>n$, set $k\leftarrow 1$. Then set $j\leftarrow k+1$.
  \item[\emph{Step 4.}]\ Judge whether both jobs $j$ and $k$ are on the same machine, or neither of them is on the critical machine $i_c$. If it is true, $j\leftarrow j+1$, go to Step 2.
  \item[\emph{Step 5.}]\ Interchange job $j$ and $k$ in solution $\pi$ and solve the single machine scheduling problem under any scenario for each machine to construct a new solution $\pi'$. Evaluate solution $\pi'$ and obtain its critical machine $i_c'$, maximum regret $R_{\max}(\pi')$.
  \item[\emph{Step 6.}]\ If $R_{\max}(\pi')<R_{\max}(\pi)$, set $\pi \leftarrow \pi'$, $R_{\max}(\pi)\leftarrow R_{\max}(\pi')$, $i_c \leftarrow i_c'$, and $l\leftarrow 0$. Go to Step 2. 
  \item[\emph{Step 7.}]\ Terminate the algorithm and output $\pi, R_{\max}(\pi)$.
\end{enumerate}
\vspace{-0.2cm}
\begin{eqnarray*}
\hline
\end{eqnarray*}


\subsection{The upper bound from the initial solution under mid-scenario}

The efficiency of the local search algorithm depends on the choice of the initial solution. Here we introduce an initial solution that could provide an upper bound for this problem. We generate an initial solution by solving the deterministic problem under the mid-scenario $s^{\frac{1}{2}}:=\{p_{ij}^{s^{\frac{1}{2}}} = \frac{\overline{p}_{ij} + \underline{p}_{ij}}{2},s_{ijk}, \forall i\in M,j,k\in N\}$. Theorem \ref{Theorem:midpoint} below shows that this solution provides an upper bound for RUPMSS. 

The proof of the upper bound is similar to \cite{xu2013robust}, but modified for this problem. Before describing the theorem, we first give some notations as below:

\begin{description} 
  \item[$\alpha_{ij}$] the percentage deviation between the upper bound and lower bound of processing time of job $j$ on machine $i$, i.e., $\alpha_{ij} = \frac{\overline{p}_{ij} - \underline{p}_{ij}}{\underline{p}_{ij}}$;
    \item[$\alpha$] the maximum percentage deviation between the upper bound and lower bound for all processing times, i.e., $\alpha = \underset{i\in M,j\in N}{\max} \alpha_{i j}$;
    \item[$\pi^s_*$] the optimal solution for the deterministic UPMSS problem under scenario $s$.
\end{description}

We can derive that

$$\overline{p}_{ij} = (1+\frac{\alpha_{ij}}{2+\alpha_{ij}}) p_{ij}^{s^{\frac{1}{2}}}, \underline{p}_{ij} = \frac{2}{2+\alpha_{ij}} p_{ij}^{s^{\frac{1}{2}}}.$$

Next we define some new scenarios:

\begin{description}
  \item[$s^{\frac{1}{2}}$] $p_{ij}^{s^{\frac{1}{2}}} = \frac{\overline{p}_{ij} + \underline{p}_{ij}}{2}$;  $s_{ijk}^{s^{\frac{1}{2}}} = s_{ijk}$;
  \item[$\overline{s}$] $p_{ij}^{\overline{s}} = \overline{p}_{ij}$; $s_{ijk}^{\overline{s}} = s_{ijk}$;
  \item[$\underline{s}$] $p_{ij}^{\underline{s}} = \underline{p}_{ij}$; $s_{ijk}^{\underline{s}} = s_{ijk}$;
  \item[$\overline{s}^{\alpha}$] $p_{ij}^{\overline{s}^{\alpha}} = (1+\frac{\alpha}{2+\alpha}) p_{ij}^{s^{\frac{1}{2}}} \geq p_{ij}^{\overline{s}}$; $s_{ijk}^{\overline{s}^{\alpha}} = (1+\frac{\alpha}{2+\alpha}) s_{ijk} \geq s_{ijk}^{\overline{s}}$;
  \item[$\underline{s}^{\alpha}$] $p_{ij}^{\underline{s}^{\alpha}} = (\frac{2}{2+\alpha}) p_{ij}^{s^{\frac{1}{2}}} \leq p_{ij}^{\underline{s}}$; $s_{ijk}^{\underline{s}^{\alpha}} = (\frac{2}{2+\alpha}) s_{ijk} \leq  s_{ijk}^{\underline{s}}$;
\end{description}

\begin{theorem}\label{Theorem:midpoint}
$\frac{2\alpha}{2 + \alpha} F^{s^{\frac{1}{2}}}_*$ is an upper bound for the maximum regret of the mid-scenario initial solution, i.e., $R_{\max}(\pi^{s^{\frac{1}{2}}}_*) \leq  \frac{2\alpha}{2 + \alpha} F^{s^{\frac{1}{2}}}_*$.
\end{theorem}

\begin{proof}
Note that the processing time and setup time of scenarios $\overline{s}^{\alpha}, \underline{s}^{\alpha}$, and $s^{\frac{1}{2}}$ are proportional to each other. For any schedule $\pi$,

\begin{equation}\label{equation:upper}
F(\pi,\overline{s}^{\alpha}) = (1 + \frac{\alpha}{2+\alpha}) F(\pi, s^{\frac{1}{2}});
\end{equation}
    
\begin{equation}\label{equation:lower}
F(\pi,\underline{s}^{\alpha}) = \frac{2}{2+\alpha} F(\pi, s^{\frac{1}{2}});
\end{equation}

\begin{equation}\label{equation:upper2}
F(\pi,\overline{s}^{\alpha}) \geq F(\pi, \overline{s});
\end{equation}

\begin{equation}\label{equation:lower2}
F(\pi,\underline{s}^{\alpha}) \leq F(\pi, \underline{s}).
\end{equation}

So there exists an identical optimal solution $\pi^*$ for deterministic problems under all the scenarios $\overline{s}^{\alpha}, \underline{s}^{\alpha}$, and $s^{\frac{1}{2}}$, i.e., there exists an optimal schedule $\pi^* = \pi_*^{s^{\frac{1}{2}}} = \pi_*^{\overline{s}^{\alpha}} = \pi_*^{\underline{s}^{\alpha}}$. Suppose $s^0$ is a worst-case scenario for $\pi^*$. First we obtain the inequalities:
\begin{equation}
    F^{\underline{s}}_* = F(\pi^{\underline{s}}_*, \underline{s}) \leq F(\pi_*^{s^0}, \underline{s}) \leq F(\pi_*^{s^0}, s^0) = F^{s^0}_*
\end{equation} 
\begin{equation}\label{equation:Rmax}
    \implies R_{\max}(\pi^*) = F(\pi^*,s^0) - F^{s^0}_* \leq F(\pi^*,\overline{s}) - F^{\underline{s}}_* 
\end{equation}
Similarly, 
\begin{equation}\label{equation:strict}
F^{\underline{s}^{\alpha}}_* = F(\pi^*, \underline{s}^{\alpha}) \leq F(\pi_*^{\underline{s}}, \underline{s}^{\alpha}) \leq  F(\pi_*^{\underline{s}}, \underline{s}) = F^{\underline{s}}_* 
\end{equation}
From (\ref{equation:upper})(\ref{equation:lower})(\ref{equation:upper2})(\ref{equation:lower2})(\ref{equation:strict}) we obtain
\begin{equation}\label{equation:finalresult}
\begin{array}{l}
     F(\pi^*,\overline{s}) - F^{\underline{s}}_* \leq F(\pi^*,\overline{s}^{\alpha}) - F^{\underline{s}^{\alpha}}_* 
     = F(\pi^*, \overline{s}^{\alpha}) - F(\pi^* , \underline{s}^{\alpha}) \\
     = (1 + \frac{\alpha}{2+\alpha})F(\pi^*,s^{\frac{1}{2}}) - \frac{2}{2+\alpha} F(\pi^*, s^{\frac{1}{2}}) 
     = \frac{2\alpha}{2+\alpha} F^{s^{\frac{1}{2}}}_* 
\end{array}
\end{equation}
From (\ref{equation:Rmax})(\ref{equation:finalresult}) we further obtain
\begin{equation}\label{equation:upperbound}
    R_{\max}(\pi^*) \leq \frac{2\alpha}{2+\alpha} F^{s^{\frac{1}{2}}}_* 
\end{equation}

Equivalently, 

\begin{equation}
    R_{\max}(\pi_*^{s^{\frac{1}{2}}}) \leq \frac{2\alpha}{2+\alpha} F^{s^{\frac{1}{2}}}_* 
\end{equation}

\end{proof}

Theorem \ref{Theorem:midpoint} makes sure that our choice of initial solutions guarantees a performance no worse than the upper bound.

\subsection{Multi-start decomposition heuristic algorithm}

Heuristic algorithms that intend to find global optimal solutions usually require some type of diversification method to avoid being stuck in the local optima \cite{marti2003multi}. One method to achieve diversity is to re-conduct the search procedure from a different initial solution once a previous neighborhood has been explored. 
The multi-start method is a framework that executes multiple times of procedure from different initial points in the solution space \cite{avalos2015efficient}, which has been used for many combinatorial optimization problems. Readers could refer to the survey paper \cite{marti2013multi}, which exhibits its flexibility and capability . 

To solve the RUPMSS, we propose a \underline{M}ulti-start \underline{D}ecomposition-based \underline{H}euristic (MDH) algorithm based on the theorems and the methods introduced above. In each iteration of the MDH method, an initial solution in the solution space is generated first. Then the shift local search and interchange local search are conducted respectively to search its neighborhood and seek improvement. After each iteration, a solution is obtained which is typically a local optimal solution, and the best solution obtained after all iterations is the output of the MDH algorithm.

In the MDH algorithm, we obtain the initial solutions by solving a series of deterministic UPMSS problems under certain scenarios. One of the initial solution is the optimal solution under the mid-scenario, which could guarantee that the maximum regret of the final solution we find cannot be worse than the upper bound given by Theorem \ref{Theorem:midpoint}. Another two specified scenarios including the upper-bound scenario (all the processing times take the upper bound value of the interval) and the lower-bound scenario (all the processing time take the lower bound value of the interval) are also considered to obtain the initial solutions.  
Then other initial solutions are obtained by solving the deterministic UPMSS under several scenarios $s$ in which the processing time is randomly chosen within its interval. Those initial solutions identical to the previous ones are not considered for the subsequent exploration.

In the MDH algorithm, we use $Init$ to denote the total number of initial solutions we generate and ${\pi_{0}}_r$ to denote the $r$-th initial solution we obtain. Similarly, $\pi^*$ denote the current best solution and $R_{\max}(\pi^*)$ represents its maximum regret. The algorithm will terminate when the running time reaches the time limit. The procedures of the MDH algorithm are presented as follows: 
 
\begin{eqnarray*}
\hline                               \\
\leftline{\emph{The Multi-Start Decomposition Heuristic Algorithm (MDH)} }         \\
\hline
\end{eqnarray*}
\vspace{-0.7cm}
\begin{enumerate}
\setlength{\itemindent}{6pt}
\setlength{\labelsep}{0pt}
  \item[\emph{Step 1.}]\  Set $r\leftarrow 0$, $R_{\max}(\pi^*)\leftarrow Big Number$. Initialize $Init$;
  \item[\emph{Step 2.}]\ Set $r\leftarrow r+1$, If $r>Init$, go to Step 6; otherwise solve a deterministic UPMSS under a scenario to obtain the initial solution ${\pi_{0}}_r$. If ${\pi_{0}}_r$ has occurred before, repeat Step 2.
  \item[\emph{Step 3.}]\  Conduct the shift local search for the solution ${\pi_{0}}_r$ to obtain a solution ${\pi_{0}}_r'$ and its maximum regret $R_{\max}({\pi_{0}}_r')$. 
  \item[\emph{Step 4.}]\  Conduct the interchange local search for the solution ${\pi_{0}}_r'$ to obtain a solution ${\pi_{0}}_r''$ and its maximum regret $R_{\max}({\pi_{0}}_r'')$. 
  \item[\emph{Step 5.}]\ If $R_{\max}({\pi_{0}}_r'')<R_{\max}(\pi^*)$, update $\pi^* \leftarrow {\pi_{0}}_r''$, $R_{\max}(\pi^*)\leftarrow R_{\max}({\pi_{0}}_r'')$, go to Step 2.
  \item[\emph{Step 6.}]\ Terminate the algorithm and output the solution $\pi^*$ and its maximum regret $R_{\max}(\pi^*)$.
\end{enumerate}
\vspace{-0.2cm}
\begin{eqnarray*}
\hline
\end{eqnarray*}

\section{Numerical Experiments}\label{section7}

In this section, we conduct numerical experiments to test the effectiveness and efficiency of the algorithms we propose. The MIP solver we use is ILOG CPLEX 12.8. The algorithms are coded with C++ and the experiments are conducted on a 64-bit Windows 10 platform with Intel i7 3.4-GHz CPU and 16.0 GB RAM.

We generate a data set whose scale is in line with the number of the tasks in a few days in the factory we investigate. The numbers of jobs and machines are combinations of $n\in \{9,12,15,20,25,30\}$, $m\in \{3,5\}$. The parameters are chosen from uniform distributions: $s_{ijk}\in [1, 10]$, $\underline{p}_{ij} \in [1, 50]$, and  $\overline{p}_{ij} \in [\underline{p}_{ij}, 2\underline{p}_{ij}]$ for any $i\in M$ and $j, k \in N$. The intervals from which the parameters take value are in coordinate with the data sets from classic literature addressing the deterministic UPMSS, e.g., Fanjul-Peyro et al.'s work \cite{fanjul2019reformulations}. We choose the processing times' lower bounds from a wilder range than the range of the setup times because we want the focus of this study, namely the uncertain processing times, to dominate the experiments. For the support of the random processing times, we choose the upper bound to be two times of the lower bounds. This implementation prevents the model from being too conservative for it overlooks the extreme outlier cases that are unlikely to occur in reality, e.g., the realisation of the processing time is ten times of the lower bound. 20 instances are generated for each combination ``$n-m$'', and a total of 240 instances are generated for the experiments. 

In the following experiments, the time limits for all the algorithms are set to 2 hours. The total number of initial solutions generated in the MDH algorithm $Init$ is set to 5 to achieve the trade off between the effectiveness and efficiency. 

We first conduct an experiment to verify the effectiveness of the mechanisms designed in the enhanced regret evaluation method. We use nine groups of small scale instances, run the MDH algorithm with one initial solution under the mid-scenario, and compare the running time of the algorithm with different mechanisms adopted in the regret evaluation process, which are represented with the following notations:
\begin{itemize}
    \item \emph{ORIG}: the basic evaluation method with no enhanced method(traverse and calculate the regret of every extreme scenario of a solution);
    \item \emph{M3}: use the only method introduced in subsection E.3; namely, only use the lower bound of the deterministic problem to avoid extreme scenarios that are impossible to be the worst-case scenario;
    \item \emph{M23}: use methods introduced in subsection E.2 and E.3; that's to say, use the lower bound of the deterministic problem and Theorem \ref{theorem:5} to eliminate extreme scenarios;
    \item \emph{M123}: use methods introduced in subsection E.1, E.2, and E.3;
    \item \emph{M1234}: use the complete procedure of the enhanced regret evaluation method.
\end{itemize}

The results are presented in Table \ref{Table:Evaluation}. Subcolumn ``\emph{time(s)}'' presents the average running time of the algorithm with the corresponding mechanisms. Subcolumn ``\emph{Num}'' reports the average number of deterministic UPMSS problems that are solved for the evaluation in the solving process.

\begin{table}[htbp]
\centering
\addtolength{\tabcolsep}{-1pt}
\caption{Comparison of different evaluation mechanisms}\label{Table:Evaluation}
{
\begin{tabular}{crrrrrrrrrrrrrrr}
\hline 
&&\multicolumn{2}{c}{\emph{ORGI}}&&\multicolumn{2}{c}{\emph{M3}}&&\multicolumn{2}{c}{\emph{M23}}&&\multicolumn{2}{c}{\emph{M123}}&&\multicolumn{2}{c}{\emph{M1234}}\\
\cline{3-4} \cline{6-7} \cline{9-10} \cline{12-13} \cline{15-16}
\emph{m}&	\emph{n}	&	\emph{time(s)}	&  \emph{num} &&	\emph{time(s)}	&  \emph{num}	&&\emph{time(s)}& \emph{num}&&	\emph{time(s)}	&  \emph{num} &&	\emph{time(s)}	&  \emph{num}	\\
\hline
	&	9	&	9.78 	&	84.90 	&&	7.22 	&	56.30 	&&	5.94 	&	44.30 	&&	3.16 	&	12.25 	&&	4.97 	&	12.05 	\\
3	&	15	&	56.23 	&	206.70 	&&	37.76 	&	133.90 	&&	35.61 	&	120.80 	&&	18.05 	&	52.30 	&&	30.71 	&	45.70 	\\
	&	20	&	264.63 	&	370.35 	&&	178.90 	&	238.10 	&&	168.82 	&	232.55 	&&	113.17 	&	149.45 	&&	162.80 	&	151.75 	\\
\\
	&	9	&	21.66 	&	114.50 	&&	14.51 	&	68.80 	&&	8.20 	&	35.20 	&&	3.55 	&	6.85 	&&	7.28 	&	7.45 	\\
5	&	15	&	325.33 	&	289.50 	&&	180.27 	&	158.00 	&&	114.99 	&	97.80 	&&	32.93 	&	22.65 	&&	26.57 	&	14.10 	\\
	&	20	&	3613.21 	&	507.00 	&&	2228.07 	&	252.40 	&&	1735.06 	&	194.00 	&&	518.24 	&	63.55 	&&	125.94 	&	42.50 	\\
\\
	&	9	&	45.34 	&	165.20 	&&	19.29 	&	103.90 	&&	7.47 	&	34.90 	&&	7.80 	&	3.20 	&&	9.26 	&	4.40 	\\
7	&	15	&	591.22 	&	323.05 	&&	309.69 	&	169.30 	&&	138.48 	&	75.20 	&&	12.15 	&	27.61 	&&	52.74 	&	17.10 	\\
	&	20	&	8625.93 	&	518.35 	&&	4600.44 	&	259.45 	&&	4164.67 	&	140.40 	&&	777.43 	&	31.15 	&&	117.33 	&	19.80 	\\
\hline
\end{tabular}}
\end{table}

From the results, we could see that from \emph{M1} to \emph{M123}, each mechanism contributes to the reduction of the running time and the number of deterministic UPMSS problems solved. As the data scale increases, the proportion of the reduced running time increases. This result also reflects that the evaluation process accounts for most of the running time of our MDH framework. As for \emph{M123} and \emph{M1234}, \emph{M1234} doesn't perform better than \emph{M123} until the data scale surpasses $m\cdot n^2 = 1200$. That is because the HABC algorithm adopted in subsection E.4 doesn't work better than the CPLEX solver for problems with a small data scale. When the data scale increases, as we can see in the cases where $n=20,m=5$ and $n=20,m=7$, HABC massively reduces the algorithm running time, so it contributes a lot to the efficiency of the regret evaluation. Overall, we believe that implementing HABC to replace the CPLEX solver suits the real needs in the industry, where the data scale is usually beyond the ability of the solvers.

To sum up, the enhanced regret evaluation method is an critical step to improve the efficiency of the MDH algorithm. With the implement of the enhancement mechanisms, the evaluation time and the overall running time of the MDH algorithm could be reduced significantly.

\begin{figure}[htbp]
\centering
\subfloat[]{\includegraphics[height=2in]{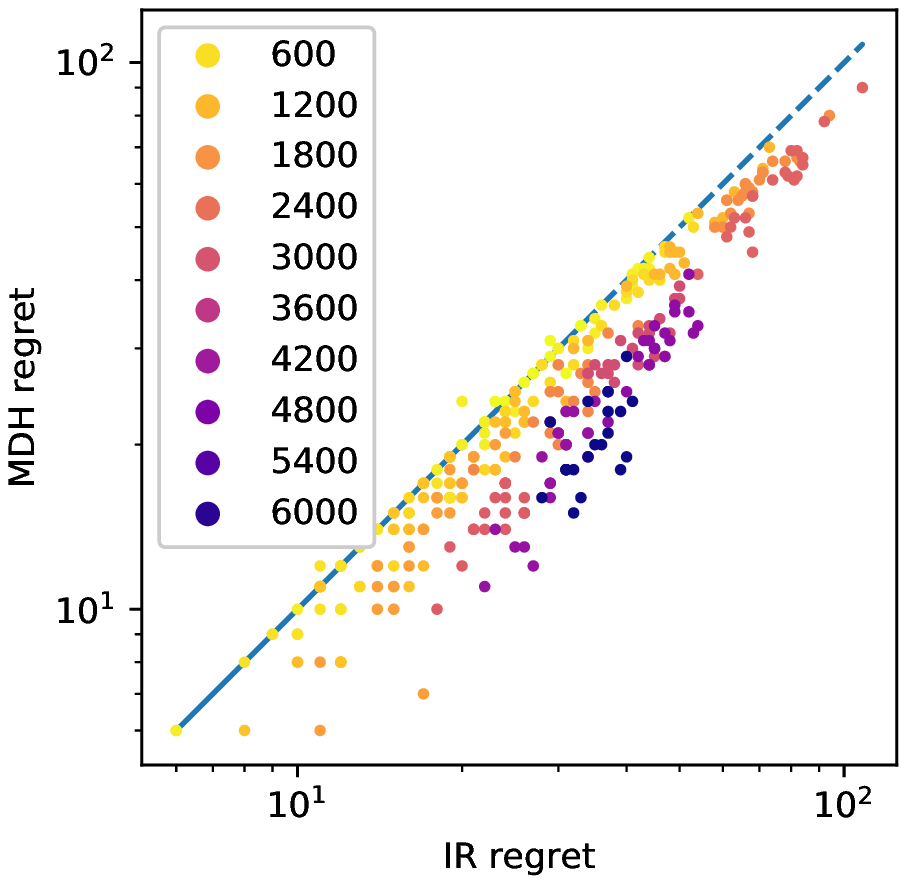}
\label{fig1}}
\hfil
\subfloat[]{\includegraphics[height=2in]{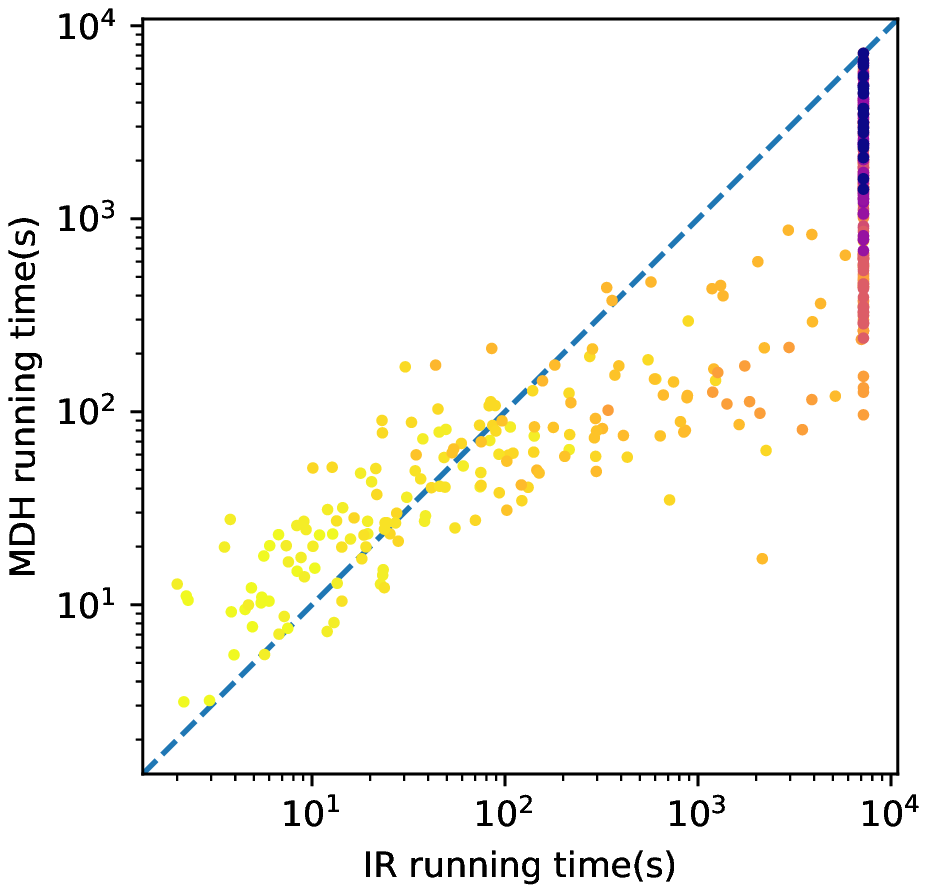}%
\label{fig2}}
\hfil
\subfloat[]{\includegraphics[height=2in]{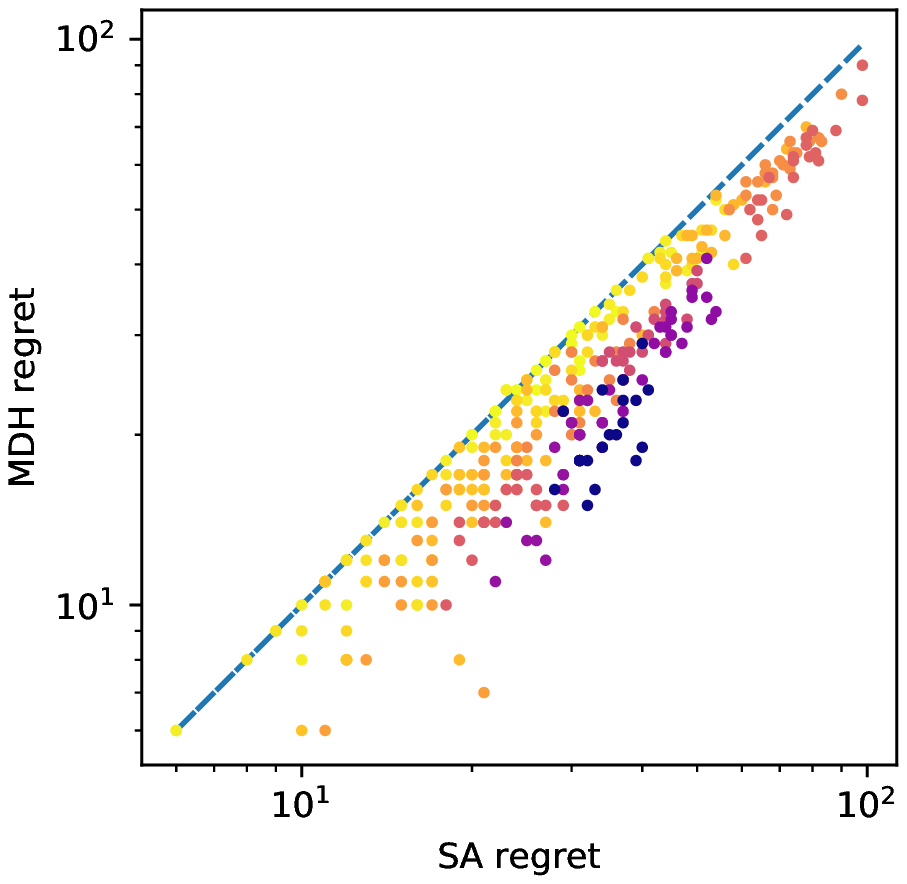}%
\label{fig3}}
\hfil
\subfloat[]{\includegraphics[height=2in]{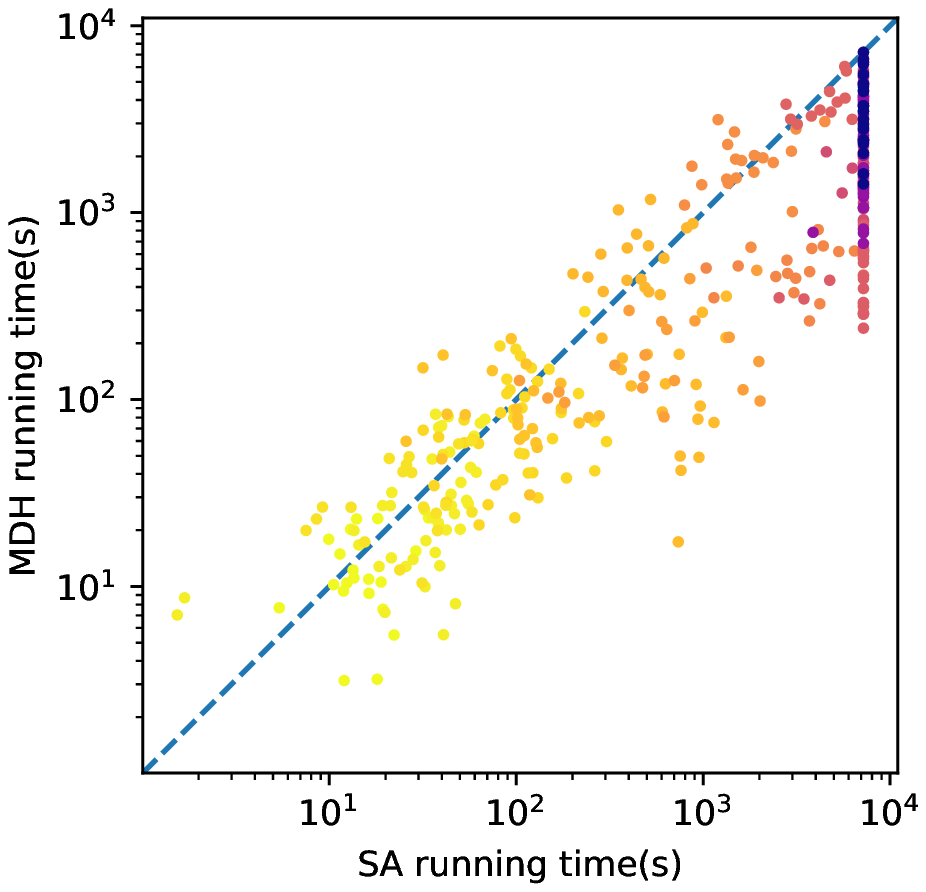}%
\label{fig4}}
\caption{Comparison between MDH, IR, and SA.  (a) Regret Comarpison: MDH vs IR. (b) Time Comparison: MDH vs IR. (c) Regret Comparison: MDH vs SA. (d) Time Comparison: MDH vs SA. }
\label{fig:Comparison}
\end{figure} 

In the following experiment, we mainly compare the results of MDH with the \underline{I}terative \underline{R}elaxation exact algorithm (IR) (see Appendix \ref{iterativeRelaxation}) and the \underline{S}imulated \underline{A}nnealing algorithm (SA) proposed by \cite{xu2013robust} and \cite{feng2016robust} for the similar problems. The time limits set for all algorithms are 2 hours. 

The notations used in the table to present the results are as follows:

\begin{itemize}
    \item \emph{mid-Gap}: gap between the maximum regret of the solution by the corresponding algorithm and that of the initial solution under mid-scenario;
    \item \emph{IR-Gap}: gap between the maximum regret of the solution by the corresponding algorithm and that of the solution by the IR algorithm;
    \item \emph{time}: average running time of corresponding algorithm;
    \item \emph{\#Opt}: number of instances solved to optimality.
\end{itemize}

where \emph{mid-Gap} is calculated as follows:

$$\text{\emph{mid-Gap}}=\frac{\text{obj(Alg)}-\text{obj(Mid)}}{\text{obj(Mid)}}\times 100\%$$
and IR-Gap is calculated as follows:

$$\text{\emph{IR-Gap}}=\frac{\text{obj(Alg)}-\text{obj(IR)}}{\text{obj(IR)}}\times 100\%$$

In these formula, ``obj'' is the objective function value, i.e., the maximum regret in this case, of the corresponding algorithm and solution. Specifically, obj(Mid) is the maximum regret of the optimal solution under mid-scenario, which yields an upper bound for RUPMSS with Theorem \ref{Theorem:midpoint}.

\begin{table}[htbp] 
\centering
\caption{Numerical Results of Different Algorithms }\label{Table:ResultsComparison}
{
\begin{tabular}{ccrrrrrrrrrrr}
\hline
&&\multicolumn{3}{c}{\emph{IR}}&&\multicolumn{3}{c}{\emph{SA}}&&\multicolumn{3}{c}{\emph{MDH}}\\
\cline{3-5} \cline{7-9} \cline{11-13}
\emph{n}&	\emph{m} &	\emph{$Gap_{mid}$(\%)} & \emph{\#Opt} & \emph{time(s)} &&\emph{$Gap_{IR}$(\%)}&  \emph{\#Opt} & \emph{time(s)} && \emph{$Gap_{IR}$(\%)} &  \emph{\#Opt} & \emph{time(s)} \\
\hline
	&	3	&	-9.56 	&	20	&	5.84 	&&	3.42 	&	14	&	15.99 	&&	0.78 	&	18	&	12.45 	\\
9	&	5	&	-10.44 	&	20	&	14.85 	&&	7.35 	&	10	&	26.30 	&&	1.45 	&	18	&	16.83 	\\
	&	7	&	-8.00 	&	20	&	33.94 	&&	7.87 	&	11	&	29.48 	&&	3.58 	&	14	&	32.64 	\\
\\
	&	3	&	-6.07 	&	20	&	47.79 	&&	3.61 	&	12	&	47.70 	&&	0.70 	&	18	&	46.39 	\\
12	&	5	&	-9.47 	&	20	&	250.48 	&&	9.52 	&	9	&	112.57 	&&	0.96 	&	17	&	48.14 	\\
	&	7	&	-8.35 	&	20	&	280.80 	&&	8.14 	&	10	&	96.19 	&&	3.54 	&	14	&	92.55 	\\
\\
	&	3	&	-6.88 	&	20	&	227.93 	&&	4.91 	&	9	&	117.50 	&&	0.62 	&	17	&	118.63 	\\
15	&	5	&	-14.48 	&	19	&	1479.40 	&&	18.68 	&	7	&	719.30 	&&	1.00 	&	17	&	122.53 	\\
	&	7	&	-7.68 	&	11	&	4603.14 	&&	6.95 	&	9	&	802.43 	&&	0.91 	&	8	&	189.83 	\\
\\
	&	3	&	-7.32 	&	14	&	3379.44 	&&	5.23 	&	10	&	457.66 	&&	1.84 	&	10	&	550.17 	\\
20	&	5	&	-5.58 	&	0	&	7200.15 	&&	1.46 	&	——	&	4130.64 	&&	-6.25 	&	——	&	552.49 	\\
	&	7	&	-1.42 	&	0	&	7200.14 	&&	0.34 	&	——	&	6747.23 	&&	-6.35 	&	——	&	447.93 	\\
\\
	&	3	&	-3.98 	&	0	&	7200.00 	&&	2.00 	&	——	&	1877.93 	&&	-1.90 	&	——	&	1894.12 	\\
25	&	5	&	0.00 	&	0	&	7200.00 	&&	-0.57 	&	——	&	6995.62 	&&	-5.26 	&	——	&	1461.08 	\\
	&	7	&	0.00 	&	0	&	7200.00 	&&	0.00 	&	——	&	7200.00 	&&	-0.94 	&	——	&	856.57 	\\
\\
	&	3	&	-3.94 	&	0	&	7200.00 	&&	0.70 	&	——	&	5647.67 	&&	-3.92 	&	——	&	4377.68 	\\
30	&	5	&	0.00 	&	0	&	7200.00 	&&	-0.21 	&	——	&	7200.00 	&&	-6.80 	&	——	&	3741.72 	\\
	&	7	&	0.00 	&	0	&	7200.00 	&&	0.00 	&	——	&	7200.00 	&&	-11.03 	&	——	&	3938.84 	\\
\hline
\end{tabular}}
\end{table}

Table \ref{Table:ResultsComparison} presents the numerical results of the corresponding algorithm for the instances we generate.  
From the subcolumn ``\emph{\#Opt}'' under IR column, we could see that for those instances with the scale smaller than or equal to ``15-3'' and ``12-7'', all the instances could be solved to optimality by IR within 2 hours. For these instances, the average gap between IR algorithm and initial solution under the subcolumn ``\emph{mid-Gap}'' are all less than 11\%, which reflects that the quality of the mid-scenario initial solution we generate is fairly good. 
When the data scale surpasses $m=15,n=5$, some instances may not be solved to optimality by IR in 2 hours. When $n$ reaches 25, for combinations ``$25-3$'' and ``$30-3$'', the values under subcolumn ``\emph{mid-Gap}'' are still negative, which means that for these problem scales, the IR algorithm could still find a solution better than the initial solution. But when the number of machines grows to 5, the values under subcolumn ``\emph{mid-Gap}'' are 0, that is to say, the solution returned by the IR algorithm has no improvement over the mid-scenario initial solution. 
As for the efficiency of the IR algorithm, as the scale of the problem becomes larger, the running time of the algorithm increases rapidly. When the data scale grows to $n=20,m=5$, IR reaches the 2-hour time limit.

As for the MDH algorithm, from the subcolumn ``\emph{IR-Gap}'' we could see that for those instances combinations that IR algorithm could yield optimal solution, the gaps between MDH and IR are small. Besides, the number of instances that could be solved to optimality by MDH is close to that of IR. From the subcolumn ``\emph{time}'' for MDH algorithm, we could see that except for the combinations with 9 jobs($n=9$), the running time of MDH algorithm is less than that of IR algorithm.   
For the instance combinations that IR algorithm unable to find the optimal solutions, the average gaps between MDH algorithm and IR algorithm are negative, which indicates that the MDH could find better solutions than IR algorithm in a shorter time. Due to that IR algorithm cannot find the optimal solution for these combinations, the dash symbol ``-'' are presented in column ``\emph{\#Opt}''. 

The results of SA algorithm are also presented in Table \ref{Table:ResultsComparison}. Compared to SA algorithm, for all combinations, the MDH algorithm obtain lower gaps and larger number of optimal solutions. As for the efficiency of the algorithms, the MDH algorithm consumes less time than the SA algorithm for nearly all combinations. So the MDH algorithm we propose performances better than the SA algorithm in terms of both the efficiency and effectiveness. 

In Fig.\ref{fig:Comparison}, we use scatter plots to showcase and compare the performances of IR, SA, and MDH intuitively. Each point in the plot represents one single instance, and its coordinate is the running time(maximum regret) of the corresponding algorithms. Obviously, the points below the $y = x$ line showcase superior performance of the algorithm corresponding to the $y$-axis, and vice versa. Note that the logarithm scales on all the plots. To denote instances of different scales, first we need to define the "data scale" of one instance as $m\cdot n^2$, which is proportional to the number of the $x_{ijk}$ variables in the optimization model. As the legend shows, the points representing small-scale instances (with small $m\cdot n^2$ values) have a brighter color. Larger-scale instances are colored darker. The dotted line is $y=x$. The points on this line means the algorithms corresponding to $x$-axis and $y$-axis yield similar performances for this instance.

Fig.\ref{fig1} and Fig.\ref{fig3} are the pairwise regret comparison for MDH vs.IR and MDH vs.SA. The two graphs have similar distributions. We could see that for most instances, MDH outperforms both IR and SA, and more importantly, the larger the instance scale is, the closer the points tend to lean towards $x$-axis, which means MDH prevails more massively. 

Fig.\ref{fig2} and Fig.\ref{fig4} are the pairwise time comparison. For Fig.\ref{fig2}, we could see that IR perform better than MDH for the small-scale instances whose data scales ($m\cdot n^2$) are less than 1000. As the data scale increases, the points flow across the dotted line $y=x$, which suggests that MDH prevails gradually. The vertical line at the right side indicates the instances where IR reaches the 2-hour time limit. Even for these instances, MDH can perform relatively well within the time limit. For Fig.\ref{fig4}, we could see that MDH and SA draw for small-scale instances. But when the data scale becomes larger, the mass of points are located below the dotted line and on the right side of the graph. Similar to Fig.\ref{fig2}, there are also dense points on a vertical line where SA reaches the time limit. For these instances, MDH yields a suporior performance compared to SA.

\section{Conclusion}\label{section8}
We discuss the min-max regret robust optimization on unrelated parallel machine scheduling with sequence-dependent setup times, which originates from a high-end equipment factory in the east of China. We give the robust optimization model and propose critical properties regarding to worst-case scenario's characteristics. The ERE method is proposed to accelerate the process of finding the maximum regret of a solution. The MDH algorithm is designed based on the findings on the decomposition property and the characteristics of the shift \& interchange search methods. Numerical experiments shows (1) the ERE method we propose significantly improves the efficiency of finding the worst-case scenario and the corresponding maximum regret; (2) our algorithm yields a better performance over other exact and heuristic algorithms designed for this type of problems. In the future work, we can further explore more uncertainty conditions, e.g., the uncertainty of the setup times; the distributional robust version of this problem may also be considered, where we need to consider the worst-case distribution for the expected value of the objective; the other robust criterion, like the relative regret, is also worth examining to further avoid the conservativeness of the robust approach.

\bibliographystyle{unsrt}
\bibliography{UPMSS}

\begin{appendices}

\section{The Iterative Relaxation Exact Algorithm (IR)}\label{iterativeRelaxation}

A widely used approach for min-max robust optimization problems is the \underline{I}terative \underline{R}elaxation algorithm (IR) (see Xu et al.\cite{xu2013robust} and Feng et al.\cite{feng2016robust}). The approach starts with a subset $S^0$ and adds one additional scenario from $S$ to this subset iteratively. In every iteration, an upper bound and a lower bound are obtained, seperately, and the gap becomes smaller after every iteration. The algorithm is terminated when the gap reaches 0 and an exact optimal solution is found.

We formulate the following $h$-th relaxed model, and define the current subset of scenarios in the $h$-th iteration as $S^h$.
\begin{alignat}{1}
&  \underset{\pi}{\min}~r  \\
&C_{k}^s-C_{j}^s+V\left(1-x_{i j k}\right) \geq s_{i j k}+p_{i k}^s,\quad j\in N_{0}, j \neq k, k \in N, i \in M, s\in S^h \\
\text{s.t.} &\sum_{j \in N_{0}, k \in N, j \neq k} s_{ijk}x_{i j k} + \sum_{j \in N}p_{i j}^s y_{ij} \leq C_{\max}^s,\quad i \in M, s\in S^h \\
&C_{\max}^s \geq C_{j}^s, \quad j\in N, s\in S^h\\
&C^s_0 = 0,\quad s\in S^h\\
&C_{\max}^s - F_s^* \leq r,\quad s\in S^h\\
& x_{ijk}\in\{0,1\}, y_{ij}\geq 0, C_j^s\geq 0, C^s_{\max}\geq 0,\quad j, k \in N, i \in M, s\in S^h\\
& \text{Constraints (\ref{C2})-(\ref{C5})} \notag
\end{alignat}

Starting with $S^0 = \emptyset$, in the $h$-th iteration, we solve the $h$-th relaxed model above to obtain solution $\pi^h$. The maximum regret of $\pi$, $R_{\max}(\pi^h)$ provides an upper bound. The optimal objective value $r^h$ gives a lower bound. The worst-case scenario $s^{h+1}$ for $\pi^h$ is generated and $S^{h+1} = S^{h} \cup \{s^{h+1}\}$. The gap $R_{\max}(\pi^h) - r^h$ becomes smaller during the iteration process. When the gap becomes 0, the optimal solution is found.

As the iterations proceed, not only the number of the constraints becomes larger, but also the scale of the decision variables in the relaxed model is growing, which adds to the difficulty of iteration. So the time for this exact algorithm to find the optimal solution grows rapidly as the scale of the problem becomes larger, and it completely fails to solve large-scale cases.

\end{appendices}

\end{document}